\documentclass[11pt]{article}
\usepackage{amssymb}
\usepackage{amsthm}

\usepackage{url}
\newtheorem{theorem}{Theorem}
\newtheorem{lemma}[theorem]{Lemma}
\newtheorem{corollary}[theorem]{Corollary}

\newtheorem{claim}{Claim}

\newcommand{\F}{\mathbb F}

\newcommand{\B}{\mathcal{B}}
\begin{document}

\title{A Removal Lemma for Systems of Linear Equations over Finite Fields}
\author{Daniel Kr{\'a}l'\thanks{%
        Institute for Theoretical Computer Science (ITI),
        Faculty of Mathematics and Physics, Charles University,
        Malostransk\'e n\'am\v{e}st\'{\i}~25, 118~00 Prague, Czech
        Republic. E-mail: {\tt kral@kam.mff.cuni.cz}. Institute for
        Theoretical computer science is supported as project 1M0545
        by Czech Ministry of Education.}
        \and
        Oriol Serra \thanks{%
        Departament de Matem\`atica Aplicada IV,
        Universitat Polit\`ecnica de Catalunya. E-mail: {\tt oserra@ma4.upc.edu}.   Supported by the Catalan Research Council
        under project 2005SGR0258.}
        \and
    Llu\'{i}s Vena \thanks{%
        Departament de Matem\`atica Aplicada IV,
        Universitat Polit\`ecnica de Catalunya. E-mail: {\tt lvena@ma4.upc.edu}. Supported by the
        Spanish Research Council under project MTM2005-08990-C01-C02.}
        }
\date{}
\maketitle

\begin{abstract}
We prove a removal lemma for systems of linear equations over finite fields:
let $X_1,\ldots,X_m$ be subsets of the finite field $\F_q$ and let
$A$ be a $(k\times m)$ matrix with coefficients in $\F_q$ and rank
$k$; if the linear system $Ax=b$ has  $o(q^{m-k})$ solutions with
$x_i\in X_i$, then we can destroy all these solutions
by deleting $o(q)$ elements from each $X_i$.
This extends a result of
Green [Geometric and Functional Analysis 15(2) (2005), 340--376]
for a single linear equation in abelian groups
to systems of linear equations. In particular,
we also obtain an analogous result for systems of equations over integers,
a result conjectured by Green.
Our proof uses the colored version of the hypergraph Removal Lemma.
\end{abstract}

\section{Introduction}

In 2005, Green~\cite[Theorem~1.5]{green05} proved the so-called
Removal Lemma for abelian groups. It roughly says that if a linear
equation over an abelian group has not many solutions one can
delete all the solutions by removing few elements. This Removal
Lemma for groups has its roots in the well--known Triangle Removal
Lemma  of Ruzsa and Szemer\'edi~\cite{ruzszem76}  (see
also~\cite{simonovits} for generalizations and many applications
of this important result in combinatorics) which roughly says that
if a certain graph has not many triangles, then they are supported
over not many edges.

In~\cite{ksv08}, the authors gave a purely combinatorial proof, by
using the Removal Lemma for graphs, of the algebraic version of
the Removal Lemma for linear equations which allows for an
extension of the result to non-abelian groups. In the same paper,
the authors considered some extensions of the result to systems of
equations in abelian groups which could be proved along the same
lines. However to extend the result to general linear systems, the
graph representation used in the mentioned paper presented serious
limitations. Instead, the extensions to hypergraphs of the removal
lemma, which have been recently proved by Nagle, R\"odl,
Schacht~\cite{rodl06}, Gowers~\cite{gowers} or Tao~\cite{tao},
seem to be the natural tool to achieve this goal.

Our main result is the following:

\begin{theorem}[Removal Lemma for systems of equations] \label{th:rm-syseq}
Let $F=\mathbb{F}_q$ be the finite field of order $q$.
Let $X_1,\ldots,X_m$ be subsets of $F$, $A$ a $(k\times m)$ matrix
with coefficients in $F$ whose rank is $k$ and $b$ a $k$-dimensional
vector over $F$.

If there are $o(q^{m-k})$ solutions of the system $Ax=b$
with $x_i\in X_i$, then there exist sets $X_1',\ldots,X_m'$
with $|X_i\setminus X_i'|=o(q)$ such that
there is no solution to the system $Ax=b$ with $x_i\in X_i'$.
\end{theorem}

So, if a linear system  has not many solutions, then these solutions
are supported by not many elements. Since this is the first time
we use the little $o$ notation, let us be more precise here:
Theorem~\ref{th:rm-syseq} asserts that
for every $\varepsilon>0$, $k$ and $m$,
there exists $\delta>0$ such that if the number of solutions
is at most $\delta q^{m-k}$, all the solutions can be destroyed
by removing at most $\varepsilon q$ elements from each of the sets $X_i$.
The value of $\delta$ depends only on $\varepsilon$, $k$ and $m$,
in particular, it is independent of $q$.

Theorem~\ref{th:rm-syseq} implies an analogous result in the
framework where solutions from subsets $X_i\subset [1,N]$ of a
linear system in the integers are sought, a result conjectured by
Green~\cite[Conjecture 9.4]{green05}.

Independently of us, Conjecture 9.4 from~\cite{green05}
was proved by Shapira~\cite{bib-shapira} whose method
also yields a different proof of Theorem~\ref{th:rm-syseq}.
Shapira's proof also uses the colored version of the hypergraph
Removal Lemma (Theorem~\ref{th:rmhc}) as our proof does.
However, his proof involves $O(m^2)$-uniform hypergraphs and
our proof involves $(k+1)$-uniform hypegraphs. At the high
level, the two proofs follow the idea of reducing
the problem to a hypergraph problem, but the particular
ideas used to reduce the problem differ a lot.

Let us also mention that the conclusion of
Theorem~\ref{th:rm-syseq} can be proven substantially easier if we
assume that every $k$ columns of the matrix are linearly
inpendent; we have reported on this result in~\cite{ksv08_2}.

As an example of application of Theorem~\ref{th:rm-syseq},
we have the following.

\begin{corollary} \label{coro:ap}
If a subset $X\subset \mathbb{F}_{q}$, $q=p^n$, $p\geq k\geq 3$, has
$o(q^2)$ arithmetic progressions of length $k$, then the set $A$ has
$o(q)$ elements.
\end{corollary}

The proof of Theorem~\ref{th:rm-syseq} follows 
the
main idea of the proof presented in~\cite{ksv08}. As we have already
mentioned, we use the edge-colored version of the hypergraph Removal
Lemma which follows from a more general result of Austin and
Tao~\cite[Theorem 2.1]{at}.

\begin{theorem}[Austin and Tao~\cite{at}]\label{th:rmhc}
Let $H$ be an edge-colored $(k+1)$-uniform hypergraph with $m$
vertices. Let $K=\left(V,E\right)$ be an edge colored $(k+1)$-uniform
$t$-partite hypergraph with $M$ vertices.
If the number of copies of $H$ in $K$ (preserving the colors of the edges)
is $o(M^m)$, then there is a set $E'\subset E$ of size $|E'|=o(M^{k+1})$
such that $K'=(V,E\setminus E')$, as an edge-colored hypergraph, is $H$-free.
\end{theorem}

\section{Main result and its proof} \label{sec:proof}

In this section, we present the proof of Theorem~\ref{th:rm-syseq}
except for an auxiliary lemma (Lemma~\ref{lem:mat-c-rk}) whose
proof is given in the next section.
We first explain the main steps of the proof.

\subsection{Outline of the proof} \label{sec:outline}

We associate to the system $Ax=b$, where $A$ has size $k\times m$,
an edge-colored $(k+1)$-uniform hypergraph $H$ with $m$ edges and
$m$ vertices. We shall construct a large $m$-partite $(k+1)$-uniform
hypergraph $K$ on $mq$ vertices built up with $m$ copies, $F_1, F_2,
\ldots,F_m$, of the field $F=\F_q$. The edges of $K$ are defined in
such a way that each solution of the system corresponds to $q^k$
edge--disjoint copies of $H$ with each edge representing an element
of $X_i$. The bound on the number of solutions of our linear system
translates to the fact that $K$ contains $o(q^{m})$ copies of $H$.
By the Removal Lemma for hypergraphs, Theorem~\ref{th:rmhc}, we will
find a set $E'$ of edges with size $o(q^{k+1})$, such that, by
removing $E'$ from $K$ we delete all copies of $H$. We then apply a
pigeonhole argument to reduce the $o(q^{k+1})$ edges to $o(q)$
elements from each set $X_i$ using the fact that the $q^{k}$ copies
corresponding to the same solution are edge--disjoint.

The key point above is the construction of the auxiliary hypergraph.
Before we explain the details of this construction,
we show that we can assume the matrix $A$ to be of a certain
special form.

\subsection{Reductions of the system}\label{sec:reduction}

First observe that, by the nature of the statement of Theorem~\ref{th:rm-syseq},
there is no loss of generality in assuming that the matrix $A$
has full rank $k$. By permuting the columns and an
appropriate choice of a basis of $F^m$, the matrix $A$ can be
assumed to be of the form $A=(I_k | B)$ where $I_k$ is the identity
matrix.

Our second observation is that it suffices to prove Theorem
\ref{th:rm-syseq} for homogenous systems. Indeed, if $A$ is written
in the form $(I_k|B)$, the general statement follows by applying it
to the system $Ax=0$ once we replace the first $k$ sets
$X_1,\ldots ,X_k$ by $X'_1=X_1-b_1,\cdots , X'_k=X_k-b_k$, where
$b=(b_1,\ldots ,b_m)$.

We will make a further assumption on $A=(I_k|B)$: any two rows in
$B$ have rank $2$. Suppose on the contrary that rows $B_i$ and $B_j$
are not linearly independent, say $B_i=\lambda B_j$. This implies
that every solution of the system $Ax=0$ satisfies $x_i=\lambda^{-1}
x_j$. Therefore we can replace $X_i$ by $X'_i=X_i\cap
\left(\lambda^{-1}\cdot X_j \right)$, delete the $j$-th equation
together with the $j$-th variable and apply our theorem
in the resulting setting: the obtained system contains
one less equation and one less variable.

We may also assume that any row $B_i$ has, at least, two non-zero
entries. Otherwise the $i$'th equation would read $x_i+b_{i,j}x_j=0$
for some $j\in [k+1,\ldots,m]$. As in the preceding paragraph, we
can replace the set $X_j$ by $X'_j=X_j\cap \left(-b_{i,j}^{-1}\cdot
X_i\right)$ and consider the system obtained
by eliminating the $i$-th equation and the $i$-th variable.

Consequently, we can assume that $A$ and $b$ satisfy the following:

\begin{itemize}

\item[{\rm (i)}] $b=0$.

\item[{\rm (ii)}] The matrix $A$ has the form $A=(I_k|B)$ where $I_k$ is the
identity matrix.

\item[{\rm (iii)}] Every two rows of $B$ are linearly independent.

\item[{\rm (iv)}] Each row of $B$ has at least two non-zero entries.

\end{itemize}

Notice that the condition (iv) implies that $m\geq k+2$.

\subsection{Construction of the hypergraph $K$}
\label{sec:build-hyper}

For the construction of the hypergraph $K$ with the properties
given in Subsection~\ref{sec:outline}, we shall use an auxiliary
matrix associated to the matrix $A$ which is described in
Lemma~\ref{lem:mat-c-rk}. Before we state the lemma, let us
introduce some additional notation. If $M$ is a matrix, the $i$-th
row of a matrix $M$ is denoted by $M_i$ and $M^j$ denotes its $j$-th
column. The support of a vector $x\in F^n$, denoted by $s(x)$, is
the set of coordinates with a nonzero entry.

\begin{lemma} \label{lem:mat-c-rk}
Let $A=(I_k|B)$ be a $(m\times k)$-matrix with coefficients in
$\F_q$. There are a $(m\times m)$ matrix $C$ and $m$ pairwise
distinct $(k+1)$-subsets
 $S_1,\ldots ,S_m\subset [1,m]$ with the following
properties:
\begin{enumerate}
\item $AC=0$ \label{c:sys-full}
\item $rank(C) = m-k$ (maximal subject to the first condition).
\label{c:tot-rank}
\item For every $i$, $s (C_i)\subset S_i$ and $s(C_i)\neq \emptyset$.
\label{c:row-sup}
\item For every $i$,
      there exists a subset $S_i'\subset S_i$ with $|S_i'|=k$ such that
      the set of columns $\{C^j,\,\; j\in [1,\ldots,m]\setminus S_i'\}$
      has rank $m-k$. \label{c:part-rank}
\end{enumerate}
\end{lemma}

The proof of Lemma~\ref{lem:mat-c-rk}
is postponed to Section~\ref{sec:c-build}.

We are ready to define a suitable hypergraph representation of
the linear system along the lines described in Subsection~\ref{sec:outline}.
Let $Ax=0$ be a linear system,
where $A$ is a $k\times m$ matrix with entries in $\F_q$ satisfying
the properties (i)--(iv) at the end of Subsection~\ref{sec:reduction}.
Let $C$ be the matrix associated to $A$  and
$S_1, \ldots, S_m$ be the $(k+1)$--subsets of $[1,m]$ satisfying
the properties stated in Lemma~\ref{lem:mat-c-rk}.

First, the hypergraph $H$ is the $(k+1)$-uniform edge-colored
hypergraph with vertex set $\{ 1,2,\ldots ,m\}$ and with edges
$S_1,S_2,\ldots, S_m$, where the edge $S_i$ is colored $i$.

The hypergraph $K$ is the $(k+1)$-uniform $m$-partite hypergraph
with the vertex set $F\times [1,m]$. For every $x\in X_i$, $K$
contains an edge $\{(a_j,j), j\in S_i\}$ if and only if
$$
\sum_{j\in S_i} C_{ij}a_j=x,
$$
and this edge is colored by $i$ and labelled by $x$. Since
the support $s(C_i)$ is nonempty, $K$ contains
precisely $q^k$ edges colored by $i$ and labelled by $x$
for each $x\in X_i$ and every color class of $K$ form a simple
hypergraph.

We now show that the hypergraphs $K$ and $H$ have the properties
given in the outline of the proof.

\begin{claim}\label{claim2} If $H'$ is a copy of $H$ in $K$,
then $x=(x_1,\ldots ,x_m)$ is a solution of the system, where $x_i$
is the label of the edge colored by $i$ in $H'$.
\end{claim}

\begin{proof}
Let $\{ (a_1,1), (a_2,2), \ldots,(a_m,m)\}$ be the vertex set of
$H'$. By the construction of $K$, it holds that $Ca=x$ where
$a=(a_1, a_2, \ldots ,a_m)$. Hence, $0=ACa=Ax$ and $x$ is a solution
of the system.
\end{proof}

\begin{claim}\label{claim1}
For any solution $x=(x_1,\ldots,x_m)$ of the system $Ax=0$ with
$x_i\in X_i$, there are precisely $q^k$ edge--disjoint copies of the
edge--colored hypergraph $H$ in the hypergraph $K$.
\end{claim}

\begin{proof}
Fix a solution $x=(x_1,\ldots,x_m)$ of $Ax=0$
with $x_i\in X_i$, $1\le i\le m$.
First,
we show that there is a copy of $H$ in $K$
in which the edge colored $i$ has label $x_i$, $1\le i \le m$.

Since the matrix $C$ has rank $m-k$ and satisfies $AC=0$,
the columns in $C$ spans the solution space in $F^m$ and thus
there is a vector $u=(u_1,\ldots,u_m)$ with $x=Cu$. In particular,
$$x_i=(C_i,u)=\sum_{j=1}^m c_{ij}u_j=\sum_{j\in S_i}c_{ij}u_j,$$
where the second equality follows from
Lemma~\ref{lem:mat-c-rk}~(\ref{c:row-sup}). Therefore, for every
$i$, the set $\{(u_j,j), j\in S_i\}$ is an edge of $K$ colored $i$
and labeled $x_i$. It follows that the edges $\{(u_j,j), j\in
S_i\}$, $i=1,\ldots m$, span a copy of $H$ in $K$. Since the kernel
of $C$ is $k$-dimensional, there are $q^k$ vectors $u$ satisfying
$x=Cu$, and each of them corresponds to a copy of $H$ in $K$. We
next verify that these $q^k$ copies are edge--disjoint.

Let $e=\{ (a_j,j), j\in S_i\}$ be an edge of $K$ colored by $i$ and
labeled $x_i\in X_i$. We show that all the $q^k$ copies of $H$ in
$K$ contain different edges colored by $i$ and labelled $x_i$. By
Lemma~\ref{lem:mat-c-rk}~(\ref{c:part-rank}), there is a subset
$S_i'\subset S_i$ of size $k$ such that $\{ C^j, j\not\in S'_i\}$ is
a set of $m-k$ linearly independent solutions of the system $Ax=0$.
Hence, we may find a vector $u=(u_1,\ldots ,u_m)$ with $x=Cu$ such
that $u_j=a_j$ for each $j\in S'_i$. With this choice, we must also
have $u_j=a_j$ for each $j\in S_i$ and the copy of $H$ associated to
this $u$ contains the edge $e$. Hence, for each edge colored $i$ and
labeled $x_i$ there is a copy of $H$ associated to $x$ in $K$ which
contains this edge. Since there are $q^k$ such edges and the same
number of copies of $H$ associated to the solution $x$, no two
copies can share the same edge colored $i$ and labelled $x_i$. By
applying the same argument to each of the colors $1,\ldots,m$, we
conclude that the $q^k$ copies of $H$ associated to the solution $x$
are edge--disjoint.
\end{proof}

We now proceed with the proof of Theorem~\ref{th:rm-syseq}.

\begin{proof}[Proof of Theorem~\ref{th:rm-syseq}]
Since the number of solutions is $o(q^{m-k})$,
by Claims~\ref{claim1} and~\ref{claim2},
$K$ contains  $o(q^m)$ copies of $H$.
By the Removal Lemma for colored hypergraphs (Theorem~\ref{th:rmhc}),
there is a set $E'$ of edges of $K$ with size
$o(q^{k+1})$ such that, by deleting the edges in $E'$ from $K$,
the resulting hypergraph is $H$-free.

The sets $X'_i$ are constructed as follows: if $E'$ contains at
least $q^k/m$ edges colored with $i$ and labelled with $x_i$, remove
$x_i$ from $X_i$. In this way, the total number of elements removed
from all the sets $X_i$ together is at most $m\cdot o(q)=o(q)$.
Hence, $|X_i\setminus X'_i|=o(q)$ as desired. Assume that there is
still a solution $x=\left(x_1,x_2, \ldots,x_m\right)$ with $x_i\in
X'_i$. Consider the $q^k$ edge--disjoint copies of $H$ in $K$
corresponding to $x$. Since each of these $q^k$ copies contains at
least one edge from the set $E'$ and the copies are edge--disjoint,
$E'$ contains at least $q^k/m$ edges with the same color $i$ and the
same label $x_i$ for some $i$. However, such $x_i$ should have been
removed from $X_i$.
\end{proof}

\section{Proof of Lemma~\ref{lem:mat-c-rk}} \label{sec:c-build}

In this section, we give an effective construction of the matrix $C$
with the properties stated in Lemma~\ref{lem:mat-c-rk}.

We first define a sequence $\B_1,\ldots ,\B_m$ of basis of the
column space of $A$ which is formed by columns of $A$.
For a base $\B_i$, $T_i=\{j\in[1,m]: \; A^j\in \B_i\}$
denotes the set of indexes of the columns of $A$ contained in $\B_i$.

We set $T_{k}=[1,k]$, i.e., $\B_{k}=\{ A^1,\ldots ,A^k\}$.
Since no row of the submatrix $B$ is the zero vector, we
may assume, up to reordering the columns from $B$, that
$A_{1,k+1}\neq 0$.

Suppose that $\B_{i}$ has been defined for some $k\le i<m$. Express
each vector $A^{j}$, $i+1\le j\le m$ in the basis $\B_{i}$, i.e.,
$A^{j}=\sum_{r\in T_{i}}b_{r,j}A^r$. Let $g(i+1)$ be the smallest
$r$ such that $b_{r,j}\neq 0$ for some $j\in [i+1,m]$. By permuting
the columns, if necessary, we can assume that $b_{g(i+1),i+1}\neq
0$. The base $\B_{i+1}$ is then obtained from $\B_{i}$ by replacing
the column $A^{g(i+1)}$ with $A^{i+1}$, i.e.,
$T_{i+1}=(T_i\setminus\{g(i+1)\})\cup\{i+1\}$. In particular, if
$i=k$, $T_{k+1}=[2,k+1]$ since $A_{1,k+1}\neq 0$.

We have now defined the bases $\B_k,\ldots,\B_m$. Observe that
$1=g(k+1)<\cdots<g(m)$. Moreover, $T_i\subset [1,i]$ for $k\le i\le
m$. We further set $\B_0=\B_m$ for our convenience.

Suppose that $\B_{i}$ is defined for some $0\le i<k$. We proceed to
define $\B_{i+1}$ in a similar manner. Let $g(i+1)$ be the smallest
index in $T_i\setminus [1,i]$ such that the corresponding
coefficient of the vector $A^i$ expressed in the base $\B_i$
is non-zero. Note that $g(i+1)$ is well-defined since the vectors
$A^1,\ldots,A^i$ are linearly independent. The base $\B_{i+1}$ is
obtained from $\B_{i}$ by replacing $A^{g(i+1)}$ with $A^{i+1}$. In
particular, it always holds $\{1,\ldots ,i\}\subset T_{i}$.
Moreover, the base $\B_k$ defined in this way coincides
with our original choice of it.

An intuitive way of understanding the basis $\B_i$ and sets $T_i$ is
as follows: the set $T_i$ is the lexicographically maximal decreasing
sequence of $k$ indices from $[i-m+1,i]$ (indices are taken modulo $m$)
such that the columns with these indices are linearly independent.
Let us prove this claim formally:

\setcounter{claim}{0}
\begin{claim}\label{claim:maxlex}
The set $T_i$ viewed as a subset of $[i-m+1,i]$ and decreasingly ordered
is the lexicographically maximal subset of $[i-m+1,i]$ such that
the columns $A^j$, $j\in T_i$,
generate the column space of $A$ (indices are taken modulo $m$).
\end{claim}

\begin{proof}
We prove the claim by induction for $i=k,\ldots,m+k-1$.
The claim holds if $i=k$ as $T_k=\{1,\ldots,k\}$ and
this is the lexicographically maximal subset of $[k-m+1,k]$.
Let $T'_{i}=\{t'_1,\ldots,t'_k\}$ be the lexicographically
maximal subset of $[i-m+1,i]$ such that $A^{t'_1},\ldots,A^{t'_k}$
are linearly independent and let $T_{i}=\{t_1,\ldots,t_k\}$.
For simplicity, assume $t_1>t_2>\ldots>t_k$ and
$t'_1>t'_2>\ldots>t'_k$.

Let $j$ be the first index such that $t_j\not=t'_j$. Clearly,
$t'_j>t_j$. Since $t_1=t'_1=i$, it holds $j>1$. Since
$A^{t'_2},\ldots,A^{t'_j}$ are linearly independent,
$\{t'_2,\ldots,t'_{j-1}\}=\{t_2,\ldots,t_{j-1}\}\subseteq
T_{i-1}\setminus \{t_2,\ldots ,t_{j-1}\}$ and $T_{i-1}$ is the
lexicographically maximal $k$-subset of $[i-m,i-1]$ corresponding to
linearly independent columns of $A$, $T_{i-1}$ must contain an
element $r$ that is larger or equal to $t'_j$. Clearly, $g(i)=r$
(otherwise, $r\in T_{i}$).

If $r=t'_j$, then the vectors $A^{t'_1},A^{t'_2},\ldots,A^{t'_j}$
are not linearly independent (as $g(i)=r$ so that $A^r$ is in the
span of $A^{t_1},\ldots ,A^{t_{j-1}}$). If $r>t'_j$, then the the
sets of vectors $A^{t_1}$, $A^{t_2},\ldots,A^{t_{j-1}}$ and
$A^{t_2},\ldots,A^{t_{j-1}}$, $A^{r}$ span the same linear space
which implies that the vectors $A^{t_2},\ldots,A^{t_{j-1}}$,
$A^{r}$, $A^{t'_j}$ are linearly independent. Consequently,
$T_{i-1}=\{t_2,\ldots,t_k\}\cup\{r\}$ is not lexicographically
maximal independent $k$-subset of $[i-m,i-1]$ (recall that
$t'_2=t_2$, \dots, $t'_{j-1}=t_{j-1}$ and $t'_j>t_j$).
\end{proof}

The next claims will be needed to check that the matrix $C$ which we
define has the properties given in Lemma~\ref{lem:mat-c-rk}.

\begin{claim}\label{claim:ginjec}
The function $g:[1,m]\rightarrow [1,m]$ is bijective.
\end{claim}

\begin{proof}
In the described construction, $i$ can only be inserted to $T_j$ if
$j=i$. If $g(r)=g(s)=i$ for a pair of distinct $r$ and $s$ (which
involves deleting $i$ twice), then an element $i$ would be deleted
twice from $T_j$ but inserted only once which is impossible.
\end{proof}

\begin{claim}
\label{claim:inotin} For every $i=1,2,\ldots ,m$, the set $T_{i-1}$
does not contain $i$. Moreover, for every $i=2,\ldots,k+1$ the set
$T_{i-2}$ does not contain $i$.
\end{claim}

\begin{proof}
The claim  holds by construction if $i\ge k+1$. Hence, we have to
focus on $i\in [1,k]$. Since the columns $A^j$, $j\in
[1,m]\setminus\{i\}$ span the column space of $A$ (by assumption
(iv) on $A$), $i\not\in T_{i-1}$ by Claim~\ref{claim:maxlex}.

To prove the second part of the statement, observe that, by
assumption (iii) on $A$ applied to rows $i-1,i$ with $2\le i\le k$,
the columns $A^j$, $j\in [1,m]\setminus\{i-1,i\}$ span the column
space of $A$. Again by Claim~\ref{claim:maxlex}, we also have
$i\not\in T_{i-2}$.

In the extremal case $i=k+1$ we  also use assumption (iv) on $A$ to
ensure that there is $j>k+1$ such that $A^1,\ldots ,A^{k-1}, A^j$ is
a base. Again by Claim~\ref{claim:maxlex}, $k+1\not\in T_{k-1}$.
\end{proof}

We can now define the matrix $C$. The $j$-th column of $C$ has its
support in $T_{j-1}\cup \{ j\}$. For $i\in T_{j-1}$, the entry $C_{ij}$ is
the coefficient of $A^i$ in the expression of $A^j$ in the base
$\B_j$:
$$
A^j=\sum_{i\in T_{j-1}} C_{ij} A^i,
$$
and $C_{jj}=-1$ (recall that, by Claim \ref{claim:inotin}, we have
$j\not\in T_{j-1}$.)

Clearly, each column of $C$ belongs to the space of solutions of the
system $Ax=0$, so that Lemma~\ref{lem:mat-c-rk}~(\ref{c:sys-full})
holds.

The submatrix of $C$ formed by the last $m-k$ columns and the last
$m-k$ rows is an upper triangular matrix with nonzero entries on the
diagonal which implies that the rank of $C$ is $m-k$.
This proves Lemma~\ref{lem:mat-c-rk}~(\ref{c:tot-rank})

We next define the family $\{ S_1,\ldots ,S_m\}$ of $(k+1)$-subsets
of $[1,m]$. By the definition of the function $g$ and of the matrix
$C$, the nonzero elements in the $j$-th column of $C$
are in the rows $[1,j]\cup [g(j),m]$ if $j\in [1,k]$ and
in the rows $[g(j),j]$ if $j\in [k+1,m]$.
Let $R\subset [1,m]\times [1,m]$ be the set of subscripts
defining this area, i.e,
$(i,j)\in R$ if and only if either $j\in
[1,k]$ and $i\in [1,j]\cup [g(j),m]$ or $j\in [k+1,m]$ and $i\in
[g(j),j]$ (see Figure \ref{fig:R} for a typical portrait of $R$.)

\begin{figure}\label{fig:R}
\setlength{\unitlength}{1mm} \begin{center}
\begin{picture}(40,40)
    \multiput(0,0)(5,0){9}{\line(0,1){40}}
    \multiput(0,0)(0,5){9}{\line(1,0){40}}
    \put(25.3,0){\line(0,1){40}}
    \put(0,15.4){\line(1,0){40}}
    \multiput(0,0)(0,1){30}{\line(1,0){5}}
    \multiput(0,35)(0,1){5}{\line(1,0){5}}
    \multiput(5,30)(0,1){10}{\line(1,0){5}}
    \multiput(10,25)(0,1){15}{\line(1,0){5}}
    \multiput(15,20)(0,1){20}{\line(1,0){5}}
    \multiput(20,15)(0,1){25}{\line(1,0){5}}
    \multiput(25,10)(0,1){30}{\line(1,0){5}}
    \multiput(30,5)(0,1){30}{\line(1,0){5}}
    \multiput(35,0)(0,1){20}{\line(1,0){5}}
    \multiput(5,0)(0,1){25}{\line(1,0){5}}
    \multiput(10,0)(0,1){15}{\line(1,0){5}}
    \multiput(15,0)(0,1){10}{\line(1,0){5}}
    \multiput(20,0)(0,1){5}{\line(1,0){5}}
\end{picture}
\end{center}
\caption{An example of the area $R$ in matrix $C$ which corresponds
to the permutation $g(1,2,3,4,5,6,7,8)=(3, 4,6,7,8,1,2,5)$.}
\end{figure}
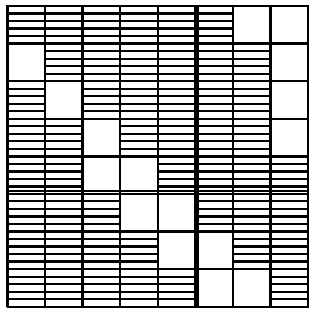

When reading off this area in the matrix by rows we get the subsets
$S_i$, namely,
$$S_i=\left\{
        \begin{array}{ll}
          g^{-1}([1,i])\cup [i,k], & i\in [1,k] \\
          g^{-1}(T_{i-1})\cup \{ i\}, & i\in [k+1,m].
        \end{array}
      \right.
$$
By the definition of $g$, the support of the row $C_i$ is contained
in $S_i$ for every $i\in [1,m]$ and none of the rows is zero (the
entry in the main diagonal is $-1$).

Let us observe that $|S_{i}|=k+1$. It follows from the definition of
$g$ that $g(i)\not\in T_{i-1}$ and, for $i\in [1,k]$,  that the sets
$g^{-1}([1,i])$ and $[i,k]$ are disjoint. Since $g$ is a bijection
(Claim \ref{claim:ginjec}) the sets $S_i$ have cardinality $k+1$ for
each $i\in [1,m]$.

Let us now show that the sets $S_i$ are pairwise distinct. Recall
that the region $R$ contains in column $j\in [1,k]$ the rows
$[1,j]\cup [g(j),m]$. It follows from the second part of
Claim~\ref{claim:inotin} that $j\not\in T_{j-2}$ for
$j=2,...,k+1$, which implies   $g(j-1)>j$.  Hence $S_j$ does not
contain $j-1$ but it does contain $j$. On the other hand,   the
column $j\in [k+1,m]$ contains in  region $R$ the rows $[g(j),j]$,
so again $S_j$ contains $j$ but does not contain $j-1$. Let
$j<j'$. If $j'\le k$ then $\{j'-1,j'\}\subset [j,k]\subseteq S_j$,
which implies $S_j\neq S_{j'}$. If $j'>k$ then, either $j'\not\in
S_j$ or, as $g$ is increasing in $[k+1,m]$, $\{j'-1,j'\}\subset
S_j$, which again implies $S_j\neq S_{j'}$.

In order to prove the last part of  Lemma \ref{lem:mat-c-rk}, we
show that the columns $\{C^j,\; j\not\in S_i\}$ form a set of
$m-k-1$ linearly independent vectors. Together with
Lemma~\ref{lem:mat-c-rk}~(\ref{c:row-sup}) this fact implies
Lemma~\ref{lem:mat-c-rk}~(\ref{c:part-rank}) and completes the proof
of the Lemma.

Let $C'=\{ C^j:\; j\not\in S_i\}$ be the submatrix of $C$ formed by
the columns with indices not in $S_i$. We divide this matrix
into four parts: the upper left $UL=\{ C_{rs}:\; r<i,\; s\in
[1,i]\setminus S_i\}$ formed by the first $i-1$ rows of $C$ and
the columns with index at most $i$,
the upper right $UR=\{ C_{rs}:\; r<i,\; s\in [i+1,m]\setminus S_i\}$
formed by the same rows and the remaining columns,
the lower right $LR=\{ C_{rs}:\; r\ge i,\; s\in
[1,i]\setminus S_i\}$ formed by the last $m-i+1$ rows and the
columns with index at most $i$ and the lower left $LR=\{ C_{rs}:\;
r\ge i,\; s\in [i+1,m]\setminus S_i\}$ with the remaining entries.

By our construction of the matrix $C$, $UR$ is an all-zero matrix,
while, as discussed in the proof of
Lemma~\ref{lem:mat-c-rk}~(\ref{c:tot-rank}),
the columns $C^j$ with $j\in [i+1,m]\setminus S_i$ are linearly
independent. On the other hand, again by the construction of $C$,
$UL$ is an upper triangular matrix. It follows that the columns of
$C'$ are linearly independent. The proof of Lemma~\ref{lem:mat-c-rk}
is now finished.


\begin{thebibliography}{9}

\bibitem{at} T.~Austin and T.~Tao, On the testability and repair of hereditary hypergraph
properties, arxiv0801.2179v1.


\bibitem{gowers}
W.~T.~Gowers.
\newblock Hypergraph regularity and the multidimensional Szemer\'edi
theorem.
\newblock Submitted, preprint avaliable at: \url{http://arxiv.org/abs/0710.3032v1}

\bibitem{green05}
B.~Green.
\newblock A Szemer\'edi-type regularity lemma in abelian
groups, with applications.
\newblock {\em Geometric and Functional
Analysis} 15(2) (2005), 340--376.

\bibitem{horn85}
R.A.~Horn, C.A. Johnson.
\newblock{\em Matrix Analysis.}
\newblock Cambridge University Press, Cambridge, 1985.


\bibitem{ksv08}
D. Kr{\'a}l', O. Serra, L. Vena.
\newblock A combinatorial proof of
the Removal Lemma for groups.
\newblock submitted (2008)

\bibitem{ksv08_2} D. Kr{\'a}l', O. Serra, L. Vena.
\newblock A removal lemma for linear systems over finite fields.
\newblock Proc. VI Jornadas Matem\`atica Discreta y
Algo´r\'{\i}tmica, Ediciones y Publicaciones de la UdL, 2008,
417--424.

\bibitem{simonovits}
J. Koml\'{o}s, M. Simonovits, Szemer\'edi's regularity lemma and its
applications in graph theory. Combinatorics, Paul Erd\H{o}s is
eighty, Vol.2 (Keszthely, 1993), 295-352, Bolyai Soc. Math. Stud.,
2, J\'anos bolyai Math Soc., Budapest, 1996.


\bibitem{rodl06}
B.~Nagle, V.~R\"odl, M.~Schacht.
\newblock The counting lemma for regular k-uniform hypergraphs.
\newblock {\em Random Structures Algorithms} 28 (2006), no. 2, 113--179.


\bibitem{ruzszem76}
I.Z.~Ruzsa, E.~Szemer\'edi.
\newblock Triple systems with no six points carrying three triangles.
\newblock {\em Combinatorics (Proc. Fifth
Hungarian Colloq., Keszthely}, 1976), Vol. II,  pp. 939--945,
Colloq. Math. Soc. J\'anos Bolyai, 18, North-Holland, Amsterdam-New
York, 1978.

\bibitem{bib-shapira}
A.~Shapira.
\newblock A proof of Green's conjecture regarding the removal properties of sets of linear equations.
\newblock Submitted, available as arXiv:0807.4901v2 [math.CO].


\bibitem{tao}
T.~Tao.
\newblock A variant of the hypergraph removal lemma.
\newblock {\em J. Combin. Theory Ser. A}, to appear.


\end{thebibliography}
\end{document}